\documentclass{article}
\usepackage[margin = 25mm]{geometry}
\usepackage{graphicx} 
\usepackage{tikz-cd} 
\usepackage{float} 
\usepackage{amsmath, amssymb, amsfonts, amscd, amsthm}
\usepackage{mathtools}
\newtheorem{theorem}{Theorem}[section]

\newtheorem{lemma}[theorem]{Lemma}
\counterwithin{figure}{section}

\DeclareMathOperator{\Hom}{Hom}

\newcommand{\GG}{\mathbb{G}}
\newcommand{\RR}{\mathbb{R}}
\newcommand{\ZZ}{\mathbb{Z}}

\newcommand{\Proj}{\mathrm{Proj}}
\newcommand{\git}{\mathbin{/\mkern-6mu/}}

\usepackage{authblk}

\title{The CompGIT package: a computational tool for Geometric Invariant Theory quotients}

\author{Robert Hanson\thanks{rh24058@essex.ac.uk}}
\author{Jesus Martinez-Garcia\thanks{jesus.martinez-garcia@essex.ac.uk}}
\affil{University of Essex}

\date{June 2025}

\begin{document}

\maketitle
\begin{abstract}
    We describe {\verb|CompGIT|}, a SageMath package to describe Geometric Invariant Theory (GIT) quotients of projective space by simple groups. The implementation is based on algorithms described by Gallardo--Martinez-Garcia--Moon--Swinarski \cite{GMGMS}. In principle the package is sufficient to describe any GIT quotient of a projective variety by a simple group --- in practice it requires that the user can construct an equivariant embedding of the polarised variety into projective space. The package describes the non-stable and unstable loci up to conjugation by the group, as well as describing the strictly polystable loci. We discuss potential applications of the outputs of {\verb|CompGIT|} to algebraic geometry problems, a well as suggesting directions for future developments. 
\end{abstract}

\section{Introduction}
Understanding quotients of projective varieties by group actions is a fundamental problem in mathematics that goes back to the 19th century German algebra and geometry school and Hilbert's 14th problem \cite{HilbertProblemList}. Such quotients are of interest in ring theory and algebraic geometry, but also find applications in fields as diverse as graph theory, coding theory, dynamical systems, and computer vision \cite{DerksenKemper}. Within algebraic geometry, this often comes in the form of a Geometric Invariant Theory (GIT) quotient --- an essential tool for constructing \textit{moduli spaces} that parametrise objects such as algebraic curves \cite{Deligne-Mumford}, K-polystable Fano varieties \cite{Liu-Xu, Gallardo-JMG-cubics, Gallardo-JMG-Spotti, JMG-Papazachariou-Zhao, 8authors} and K3 surfaces \cite{sexticsguy, Gallardo-JMG-Zhang, Laza}.

\verb|CompGIT| is a Sagemath package \cite{sagemath} which aims to significantly simplify the description of GIT quotients of the form $\mathbb P^n\git G$ where $G$ is a simple group. In principle --- but unfortunately not always in practice --- such quotients can recover any $G$-structured GIT problem. The code in \verb|CompGIT| is based on algorithms and code that first appeared in \cite{GMGMS}. 

\subsection{GIT quotients} 
\label{sec:GIT quotients}
We recall the fundamentals of GIT and their role in our code. Let $k$ be an algebraically closed field of characteristic $0$. If a projective variety $X=\Proj(R)$ is acted upon by a linear algebraic group $G$, we want to consider the quotient $X/G$. For this quotient to be well behaved in general, it is necessary to assume that $G$ is reductive \cite{counter-example}. In Hilbert's approach we consider a $G$-linearised very ample line bundle $L$ so that $R=\bigoplus_m H^0(X, L^{\otimes m})$ and consider the invariant subring
$$R^G=\bigoplus_m H^0(X, L^{\otimes m})^G\subseteq R.$$

If $R^G$ is finitely generated (which is the case when $G$ is reductive), then the inclusion $R^G\subseteq R$ induces a rational map $\pi\colon X \dashrightarrow X\git_{L}G \eqqcolon \Proj(R^G)$. The central characteristic of GIT is that $X\git_LG$ is projective variety, called the \emph{GIT quotient} of the $G$-space $X$ \cite{Mumford}. 

A geometric description of the quotient is determined by the following stability conditions (see precise definitions in \cite{GMGMS}, cf. \cite{Mumford}). The \emph{semistable locus} $X^{ss}\subset X$ is the open subset where $\pi$ is defined. The \emph{unstable locus} is the complementary closed subset $X^{us}=X\setminus X^{ss}$. The \emph{stable locus} $X^s\subseteq X^{ss}$ is the subset of all semistable closed orbits with finite stabilisers. The \emph{polystable locus} $X^{ps}$ is the subset of all orbits that are closed in the semistable locus. Thus we have the natural inclusions $X^s\subseteq X^{ps}\subseteq X^{ss}$. Then, by construction, the polystable orbits are in one-to-one correspondence with points in $X\git_LG$, which is why, in geometric applications, we are primarily interested in finding stable orbits and strictly polystable orbits. However, finding GIT semistable orbits has some interest too, e.g. when working modulo S-equivalence, or when describing the semistable substack of some moduli stack \cite{AHLH}.  

As $L$ is very ample it determines an embedding $X\subseteq \mathbb P^N$ for $N=h^0(X, L)-1$, so that $L=O_X(1)$. Being $G$-linearised implies that we have an embedding $G<\mathrm{Aut}(\mathbb P^N)$ as a subgroup, so that the action of $G$ on $\mathbb P^N$ restricts to the action of $G$ on $X$. Thus, we can define GIT (semi/poly)stability for $G$-orbits of $\mathbb P^N$ and, in fact, a point $p\in X\subset \mathbb P^N$ is $G$-(semi/poly)stable as point of $X$ if and only if it is $G$-(semi/poly)stable as a point of $\mathbb P^N$. Thus we have the identifications
\begin{align*}
    X^{ss}&=X\cap \left(\mathbb P^{N}\right)^{ss}, & X^{ps}&=X\cap \left(\mathbb P^{N}\right)^{ps}, & X^{s}&=X\cap \left(\mathbb P^{N}\right)^{s}.
\end{align*}
To solve a GIT problem, one may then simply replace $X$ by $\mathbb P^N$, determine the (semi/poly)stable orbits and then restrict to $X$. This is the approach that the computations of \verb|CompGIT| enables. In practice, given arbitrary $X$, it may be difficult to solve the required $G$-equivariant embedding problem, although many applications do come with natural embeddings, for instance when $X\subseteq\mathbb P^k$ is a hypersurface. Having a solution for $\mathbb P^N$ also works well to describe \'etale covers of $\mathbb P^N$ branched at hypersurfaces, since the moduli of hypersurfaces can be interpreted as the moduli of \'etale covers of $\mathbb P^N$ branched at a hypersurface. 

Further considerations need to be made with regards to the group $G$. The algorithms presented in \cite{GMGMS}, which \verb|CompGIT| implements, are optimised for a simple group $G$, making use of the rich action of the Weyl group, which simplifies both outputs and computations. In \cite{GMGMS}, the authors explain how a simplification of the algorithms can be obtained for general reductive group $G$ (e.g. torus). Since often in applications $G$ is simple, we make that assumption in the implementation of \verb|CompGIT|. 

In view of the above, we assume that $X=\mathbb P^N=\mathbb P(V)$, where $V$ is a vector space and $G$ is a simple group acting linearly on $X$. In other words, $V$ is a $G$-representation. The package \verb|CompGIT| computes:
\begin{itemize}
    \item  The unstable locus $X^{us}=X\setminus X^{ss}$.
    \item The \emph{non-stable locus} $X^{ns}=X\setminus X^{s}$
    \item The \emph{strictly polystable locus} $X^{sps}=X^{ps}\setminus X^{s}$ with respect to a maximal torus $T$ of $G$.
\end{itemize}

The details of the algorithms in \verb|CompGIT| can be found in \cite{GMGMS} and we omit them here. This document serves as a user manual for the package, with an emphasis on concrete examples of computations and suggestions for future improvements, in the hope that they attract interest from the community. Further examples and outputs are also available on the \verb|GitHub| homepage of our package \cite{CompGIT}. 

One way to explore \verb|CompGIT| is to go through the examples below and the ones online and refer back to the text below for additional explanations. While the code in the examples is mostly self-explanatory, some basic familiarity with \verb|SageMath| is helpful, as it is to have some understanding of the geometry behind the examples and GIT more generally.

\subsection{Related implementations} 
There have been other attempts at describing the unstable loci computationally, most notably by A'Campo and Popov \cite[Appendix C]{DerksenKemper} using the Computer Algebra Systems (CASs) \verb|pari-gp| and \verb|LiE|. Mathematically, the A'Campo-Popov implementation focuses on the affine cone (the \emph{nullcone}, in their language) $\mathcal N_{G, V}$ over the unstable locus $(\mathbb P^N)^{us}=(\mathbb P(V))^{ss}$ where $V$ is the underlying affine vector space. The nullcone can be given a stratification according to `how unstable' each orbit is and it plays an important role in the theory of invariants \cite{Collingwood-McGovern}. We refer to \cite[Appendix C]{DerksenKemper} for the details, where  algorithms to determine the stratification of the Nullcone $\mathcal N_{G, V}$ are presented for any reductive group $G$, while their implementation in  \verb|pari-gp| and \verb|LiE| assumes $G$ is semisimple. To the best of our understanding, their algorithms do not consider the non-stable/semistable nor the strictly polystable loci and may take longer, since they need to distinguish within the different orbits of $(\mathbb P^N)^{ss}$. In practice, the most important difference between their work and ours is probably that \verb|Sagemath| is currently a more accessible CAS. It would be desirable if portability between  \verb|pari-gp|, \verb|LiE| and \verb|Sagemath| was explored, so that their output could be compared to ours.

\subsection{Organisation of the paper}
We give a quick summary of what follows. In \S\ref{sec:installation} we give a quick summary of the installation procedure for \verb|CompGIT|. Section \ref{sec:simple-groups} gives a quick summary of conventions and facts from simple groups needed to understand the use of \verb|Sagemath| code. In \S\ref{sec:example-and-applications}, we consider a very simple well-known example of GIT quotient (moduli of plane cubic curves) to illustrate how \verb|CompGIT| solves a GIT problem. Then, we provide a list of examples of geometric interest that \verb|CompGIT| can attack, some of which have never been considered in the literature, discussing their feasibility. Finally, in \S\ref{sec:future-developments}, we discuss possible generalisations and optimisations, in the hope that other researchers consider improving \verb|CompGIT|. At this point is worth mentioning that \verb|CompGIT| has a GNU General Public License $3.0$, a free, copyleft license. Thus, the developers welcome improvements and changes to the software and algorithms by the community.

\subsection{Acknowledgements}
Geometric Invariant Theory is a tool developed by D. Mumford in the 1960s to solve Hilbert's problem without the need of finding $G$-invariants of $R$, or generators and relations for $R^G$. Although  the exposition above omits it, determining which points rely in $X^{ss}$, $X^s$ and $X^{ps}$ relies on Mumford's numerical criterion (often known as Hilbert-Mumford's numerical criterion), which our algorithms use extensively.

\verb|CompGIT|'s implementation is based on the algorithms (and original code) in \cite{GMGMS} by Gallardo, Martinez-Garcia, Moon and Swinarski, which in turn systematises and generalises ideas in \cite{Gallardo-JMG-seminal}. We are indebted to the co-authors in \cite{GMGMS} for letting us use and modify a previous implementation of their code. We thank Fr\'ed\'eric Chapoton for many valuable comments to the code and Alastair Litterick to help us clarify some of the representation theory in this manuscript. The implementation presented here is fully compatible with \verb|Sagemath 9.x| (essentially for \verb|Python 3.x|) and expands the code in \cite{GMGMS} to include the exceptional simple groups. 

The first author was supported by EPSRC's project EP/V055399/1 and the University of Essex. The second author was partially supported by EPSRC'S project EP/V055399/1.

\section{Installation}
\label{sec:installation}
\verb|CompGIT| is a \verb|SageMath| package written in \verb|Python 3.0|. The \verb|SageMath| computer algebra system is a free open-source mathematics software system designed for applications in many areas of mathematics, including algebra, geometry and combinatorics. \verb|SageMath| builds on top of many existing packages, such as \verb|NumPy|, \verb|SciPy|, \verb|matplotlib|, \verb|Maxima|, \verb|FLINT| and \verb|R|. The \verb|SageMath| installation guide \cite{SageMath_install} contains binaries and source code --- note that building from source code has the advantage of various computer-specific optimisations.  Our {\verb|GitHub|} homepage \cite{CompGIT} contains detailed installation instructions via pip. After successful installation each session in {\verb|CompGIT|} starts with 
\begin{verbatim}
sage: from CompGIT import *
\end{verbatim}
and if using \verb|CompGIT| within a sage file, simply include
\begin{verbatim}
from CompGIT import *
\end{verbatim}
at the top of the file.

All sample code in this article assumes first the execution of this command. We also note that the prefix \textit{sage:} is a prompt provided by the {\verb|SageMath|} console that allows one to distinguish between the inputs and the outputs.

\section{Simple groups}
\label{sec:simple-groups}

A reductive linear algebraic group $G$ is called \textit{simple} when every smooth connected normal subgroup is trivial or equal to $G$. Such groups are fundamental building blocks in group theory, as demonstrated by decomposition results such as the Jordan--Hölder theorem. 

\subsection{Root systems}
\label{sec:root-systems}
The fundamental properties of semisimple groups are encoded by a set of combinatorial data known as a \textit{root system}, which are \emph{irreducible} for simple groups. To define them, we start with a choice of maximal torus $T \subset G$, a lattice of \textit{one-parameter subgroups} $N = \Hom(\GG_m, T)$, a set of \textit{characters} $M = \Hom(T, \GG_m)$ and a perfect pairing $\langle \cdot \, , \cdot \rangle : M \times N \to \ZZ$. Let $U$ be a $G$-representation. As a $T$-module, $U$ decomposes as a direct sum of eigenspaces $U_\chi=\{u\in U \, \colon \, t\cdot u = \chi(t)u\}$, where $\chi\colon T \rightarrow \mathbb G_m$ is a character. If $U_\chi\neq 0$, then  $\chi$ is called a \emph{weight}, and $u\in U_\chi$ is a \emph{weight vector}. In the special case where $U=\mathrm{Lie}(G)$, the non-zero weights $\chi$ are called the \emph{roots} of $G$. We then denote the finite set of roots by $\Phi$.  Figure \ref{fig:roots_G2} displays the root system of $G_2$. The symmetries of the root system are encapsulated by the \emph{Weyl group} $W=W(G)$, whose action admits generating regions $F \subset N$ called fundamental chambers. Note that we can read the Weyl group of $G_2$ in Figure 3.1 by noticing that we can identify the 'longer' roots with vertices of a hexagon and the 'shorter' roots with its sides (or viceversa). Thus, the Weyl group $W(G_2)=D_{6}$, the dyhedral group of order $12$.

\begin{figure}[!ht]
\label{fig:roots_G2}
\centering
\begin{tikzpicture}[scale=0.9]
\foreach\ang in {60,120,...,360}{
\draw[->] (0,0) -- (\ang:2cm);
}
\foreach\ang in {30,90,...,330}{
\draw[->, thick] (0,0) -- (\ang:3cm);
}
\node[anchor=south west,scale=0.6] at (2,0) {$\alpha$};
\node[anchor=north,scale=0.6] at (-3,2) {$\beta$};
\end{tikzpicture}
\caption{The root system of the exceptional group $G_2$. Our choice of simple roots (as per Section \ref{se:conventions} and Table \ref{tab:simple-roots}) are labeled by $\alpha$ and $\beta$.} 
\end{figure}
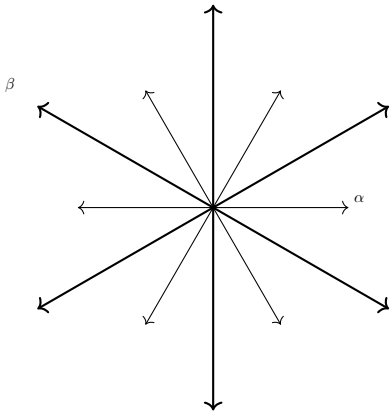

Classical works of Cartan \cite{cartan}, Chevalley \cite{chevalley} and Killing \cite{killing} show that reductive simple groups over an algebraically closed field are classified by the \textit{Dynkin type} of their root system, which takes values in one of $A_n$, $B_n$, $C_n$, $D_n$, $E_6$, $E_7$, $E_8$, $F_4$, $G_2$ where $n \geq 2$ is an integer. Weyl \cite{weyl_classical} named the Dynkin types $A_n$, $B_n$, $C_n$ and $D_n$ \textit{classical}, whereas $E_6$, $E_7$, $E_8$, $F_4$, $G_2$ are called \textit{exceptional.}

In order to understand the input and output of the programme, it is helpful to recall some additional representation theory of simple groups. Let $G$ be a simple group and $V$ a $G$-representation. Up to conjugation, $G$ has a unique Borel subgroup $B\cong T\ltimes U$ where $T$ is a maximal torus and $U$ is the unipotent radical. When $G$ is embedded into $GL_n$, we can assume it is written in matrix form and, up to conjugation, we can write $B$ as upper triangular matrices, with $T$ given as diagonal elements and $U$ given as strict upper triangular matrices. The Borel subgroup $B$ leaves some one-dimensional subspaces of $V$ invariant. If $v\in V$ is a non-zero vector of such an invariant one-dimensional subspace, then we can write $g\cdot v = t_gv$ for some $t_g\in \mathbb G_m$. This defines a group homomorphism $t\colon B \rightarrow \mathbb G_m$. The \emph{weight} $w$ of an invariant vector $v$ is the character $t_v\colon T \rightarrow \mathbb G_m$ obtained by restricting $t$ to $T$. Note we can view $w\in \mathbb Z^n$ where $n=\dim (T)$ is the rank of $G$. The subgroup $B$ acts on the eigenspaces $V_\chi$, sending them to other eigenspaces. It can be show that if $V$ is irreducible there is a unique eigenspace $V_w$ corresponding to a unique weight $w$ such that $B\cdot V_w\subseteq V_w$. The weight $w$ is called the \emph{highest} weight. Moreover, if a weight $w$ is the highest weight of an irreducible $G$-representation $V$, then $V$ is uniquely determined by $w$ up to isomorphism.

Thus, when considering the action of $G$ on $X=\mathbb P^N=P(V)$, we may assume that $V$ is a $G$-reducible representation and it is determined by a (highest) weight $w\in \mathbb Z^n$ where $n$ is the rank of $G$ and it is sufficient to give the highest weight $w$ of $V$ to determine $V$. 

\verb|Sagemath|'s standard library is already able to deal with  $G$-representations when $G$ is a simple group and \verb|CompGIT| builds on it. For instance, to introduce a $\mathrm{Spin}(8)$-representation of highest weight $\omega_3$, we would write:
\begin{verbatim}
Phi = WeylCharacterRing("C4")
representation= Phi(Phi.fundamental_weights()[3])
\end{verbatim}
In order to create a $\mathrm{Spin}(4)$-representation of highest weight $(2,0)$, one may write:
\begin{verbatim}
Sp4 = WeylCharacterRing("C2")
representation3 = Sp4(2,0);
\end{verbatim}

It may be worth to keep in mind a more geometric example. Consider $G=\mathrm{SL}_{n+1}$ acting naturally on $\mathbb P^n$. Let $V=H^0(\mathbb P^n, \mathcal O_{\mathbb P^n}(d))$ where $d\geqslant 1$, i.e. $V$ represents homogeneous polynomials of degree $d$, i.e. hypersurfaces of degree $d$ in $\mathbb P^n$. Note that the $G$-action on $\mathbb P^n$ induces a $G$-representation on $V$ and, in fact, its maximal weight is $w=(d, 0, \ldots, 0)\in \mathbb Z^{n+1}$. Let $X=\mathbb P(V)$. The GIT problem $X\git \mathrm{SL}_{n+1}$ describes equivalent hypersurfaces of degree $d$ up to $\mathrm{PGL}_{n+1}$-isomorphism.

\subsection{Implementation in CompGIT} The combinatorial characterisation of simple groups provided by root systems is both pervasive throughout group theory \cite{bourbaki} and essential to computations within {\verb|CompGIT|}. We represent various properties of the root system as attributes within the class {\verb|SimpleGroup|}. Our code utilises inputs from the {\verb|Weyl Groups|} standard package included in \verb|SageMath| --- the canonical method of representing root system data within {\verb|SageMath|}. Below are some elementary computations with the {\verb|SimpleGroup|} class applied to the root system $A_2$. 

\begin{verbatim}
sage: G=SimpleGroup("A", 2)
sage: G.group_type()
'A'
sage: G.rnk()
2        
sage: G.fundamental_chamber_generators() 
[2 1]
[1 2]
sage: G.pairing(1, [1,2])
(-1, 2)
\end{verbatim}


\subsection{Root system conventions} 
\label{se:conventions}
We specify various coordinate systems that we use to represent root system data. We closely follow the terminology outlined by Fulton and Harris \cite{fulton-harris}. The \textit{H-coordinates} on the lattice $N$ of one-parameter subgroups are defined by basis elements given by the set $\{H_i\}_{i = 1, ..., n+1}$ of $(n+1) \times (n+1)$ matrices with only one non-zero (i, i)th element of unitary size. The \textit{L-coordinates} on the set $M$ of characters are dual to the H-coordinates. For type $A$ root systems we introduce an auxiliary coordinate system given by $\{ T_i := H_i - H_{i+1} \}_{i = 1, ..., n}$, where the reduction from $n+1$ to $n$ basis coordinates is to account for the fact that type $A$ weights of a one parameter subgroup add to $0$. 

We implement the symmetries of the Weyl group with respect to certain choices that we now describe. We begin by selecting a set of \textit{simple roots} $\Delta \subset \Phi$ which are roots that cannot be written as a sum of two positive roots. To our choice of $\Delta$ the corresponding fundamental chamber $F_{\Delta}$ is then given by vectors $v \in \RR^n$ that satisfy the positivity condition $\langle \alpha , v \rangle \, \geq 0$ for all $\alpha \in \Delta$. We then define \textit{gamma coordinates} that are defined by a basis set of \textit{gamma rays} --- a class of vectors whose positive span coincides with $F_{\Delta}$. To specify our choice of $\Delta$, consider the standard Euclidean basis $e_1, ... e_n$ for $\RR^n$ and refer to the choices of simple roots provided by Table \ref{tab:simple-roots}. Our simple roots for $G_2$ are also depicted in Figure \ref{fig:roots_G2}. 

\bgroup
\def\arraystretch{1.4}

\begin{table}[!ht]
\begin{center}
\begin{tabular}{c|c}
\textbf{Dynkin type}
&
\textbf{Simple roots}
\\ \hline
$A_n$ & $e_i - e_{i-1}$ for $i = 1, ... n$
\\
$B_n$ & $e_i - e_{i-1}$ for $i = 1, ... n - 1$ and $e_n$
\\ 
$C_n$ & $e_i - e_{i-1}$ for $i = 1, ... n - 1$ and $2e_n$
\\
$D_n$ & $e_i - e_{i-1}$ for $i = 1, ... n - 1$ and $e_{n-1} + e_n$
\\
$F_4$ 
& 
$\begin{bmatrix} 0  \\  0  \\ 0  \\  1/2 \end{bmatrix}, 
\begin{bmatrix}  1  \\  0  \\ 0  \\ -1/2 \end{bmatrix}, 
\begin{bmatrix} -1  \\  1  \\ 0  \\ -1/2 \end{bmatrix}, 
\begin{bmatrix}  0  \\ -1  \\ 1  \\ -1/2 \end{bmatrix}$
\\
$E_6$ 
& 
$
\begin{bmatrix} 1 \\ 0 \\ 0 \\ 0 \\ 0 \\ -1/2 \end{bmatrix},
\begin{bmatrix} -1 \\ 1 \\ 0 \\ 0 \\ 0 \\ -1/2 \end{bmatrix},
\begin{bmatrix} 0 \\ -1 \\ 1 \\ 0 \\ 0 \\ -1/2 \end{bmatrix},
\begin{bmatrix} 0 \\ 0 \\ -1 \\ 1 \\ 0 \\ -1/2 \end{bmatrix},
\begin{bmatrix} 0 \\ 0 \\ -1 \\ 1 \\ 0 \\ -1/2 \end{bmatrix},
\begin{bmatrix} 0 \\ 0 \\ 0 \\ 0 \\ 0 \\ \sqrt{3}/2 \end{bmatrix}
$ 
\\
$E_7$
&  
$
\begin{bmatrix} 1 \\ 0 \\ 0 \\ 0 \\ 0 \\ 0 \\ -1/2 \end{bmatrix},
\begin{bmatrix} -1 \\ 1 \\ 0 \\ 0 \\ 0 \\ 0 \\ -1/2 \end{bmatrix},
\begin{bmatrix} 0 \\ -1 \\ 1 \\ 0 \\ 0 \\ 0 \\ -1/2 \end{bmatrix},
\begin{bmatrix} 0 \\ 0 \\ -1 \\ 1 \\ 0 \\ 0 \\ -1/2 \end{bmatrix},
\begin{bmatrix} 0 \\ 0 \\ 0 \\ -1 \\ 1 \\ 0 \\ -1/2 \end{bmatrix},
\begin{bmatrix} 0 \\ 0 \\ 0 \\ 0 \\ -1 \\ 1 \\ -1/2 \end{bmatrix},
\begin{bmatrix} -1/2 \\ -1/2 \\ -1/2 \\ -1/2 \\ -1/2 \\ -1/2 \\ \sqrt{2}/2 \end{bmatrix}
$
\\
$E_8$ 
&  
$
\begin{bmatrix} 1 \\ 0 \\ 0 \\ 0 \\ 0 \\ 0 \\ -1/2 \\ 0 \end{bmatrix},
\begin{bmatrix} -1 \\ 1 \\ 0 \\ 0 \\ 0 \\ 0 \\ -1/2 \\ 0 \end{bmatrix},
\begin{bmatrix} 0 \\ -1 \\ 1 \\ 0 \\ 0 \\ 0 \\ -1/2 \\ 0 \end{bmatrix},
\begin{bmatrix} 0 \\ 0 \\ -1 \\ 1 \\ 0 \\ 0 \\ -1/2 \\ 0 \end{bmatrix},
\begin{bmatrix} 0 \\ 0 \\ 0 \\ -1 \\ 1 \\ 0 \\ -1/2 \\ 0 \end{bmatrix},
\begin{bmatrix} 0 \\ 0 \\ 0 \\ 0 \\ -1 \\ 1 \\ -1/2 \\ 0 \end{bmatrix},
\begin{bmatrix} 0 \\ 0 \\ 0 \\ 0 \\ 0 \\ 1 \\ -1/2 \\ -1 \end{bmatrix},
\begin{bmatrix} 0 \\ 0 \\ 0 \\ 0 \\ 0 \\ 0 \\ -1/2 \\ 0 \end{bmatrix}
$
\\
$G_2$ 
& $\begin{bmatrix} 1 \\ -3/2 \end{bmatrix}, \begin{bmatrix} 0 \\ \sqrt{3}/2 \end{bmatrix}$
\end{tabular}
\end{center}
\caption{Choice of simple roots for simple groups}
\label{tab:simple-roots}
\end{table}
\egroup

\section{Finding the unstable and non-stable loci}
\label{sec:example-and-applications}
Having explained the conventions we use to represent group data, we can now demonstrate how \verb|CompGIT| solves GIT problems of geometric interest. We first run a toy example: the moduli space of plane cubics in $\mathbb P^2$ (\S\ref{sec:comp-example}), to familiarise the reader with \verb|CompGIT| and demonstrate how one reads and interprets the outputs (\S\ref{sec:comp-interpret}). We then discuss more complex applications (\S\ref{sec:applications}) that either recover problems covered in the literature or provide new problems of geometric interest. 

\subsection{Moduli of plane cubics}
\label{sec:comp-example}
A plane cubic curve in $\mathbb P^2$ is in one-to-one correspondence with cubic forms in $3$ variables (prior to any identification via isomorphisms). Thus, their parameter space can be identified with the vector space $V=H^0(\mathbb P^2, \mathcal O_{\mathbb P^2}(3))$. An element $f\in V$ can be seen as as a homogeneous polynomial
$$f=\sum a_Ix^I$$
where the sum is over all $I=(d_0, d_1, d_2)\in \mathbb Z_{\geqslant 0}^3$ such that $d_0+d_1+d_2=3$ and such that $x_I=x_0^{d_0}x_1^{d_1}x_2^{d_2}$ and $a_I\in \mathbb C$. The natural action of $G=\mathrm{SL}_3(\mathbb C)$ on $\mathbb P^2$ (by acting on the coordinates $(x_0, x_1, x_2)$ induces an action on $V$ (where $g\in \mathrm{SL}_3(\mathbb C)$ sends $f(x_0, x_1, x_2)$ to $f(g\cdot (x_0, x_1, x_2))$). Thus $V$ is a $G$-representation and we have that $\mathbb P(V)\git G\cong \mathbb P(V)\git \mathrm{Aut}(\mathbb P^2)= \mathbb P(V)\git \mathrm{PGL}_3$. From the exposition in \S\ref{sec:root-systems}, it follows that $x_0^3$ is a weight. In fact, under the right choice of coordinates for $B\cong T\ltimes U$ (e.g. by choosing $T$ to be the diagonal entries in $\mathrm{SL}_3$ and $U$ to be the lower-diagonal entries), $x_0^3$ is the highest weight. Recall that the highest weight uniquely determines the $G$-representation. After loading \verb|CompGIT| (see \S\ref{sec:installation}), we can create this representation by determining the type of the group \verb|A2| and the weight \verb|(3,0,0)|
\begin{verbatim}
sage: Phi = WeylCharacterRing("A2")
sage: representation= Phi(3,0,0)
\end{verbatim}
The next line is to initialise the problem:
\begin{verbatim}
sage: P = GITProblem(representation)
\end{verbatim}
The object \verb|P|, of class \verb|GITProblem|, has all the information required to `solve' the problem. It contains a lot of information that can be `requested' from the class via different class methods. It is possible that some of these can be used to obtain further information from the quotient $\mathbb P(V)\git G$ (e.g. information about its cohomology, `a la Kirwan' \cite{Kirwan}). The information inside \verb|P| includes a list of one-parameter subgroups that suffices, via the Hilbert-Mumford criterion \cite[Theorem 2.1]{Mumford}, to determine its stability.

We can now ask \verb|CompGIT| to find all non-stable elements. In this example, this is fast but in some cases this can take a long time (see \cite[Table 1]{GMGMS} for a discussion of times for a large number of examples).
\begin{verbatim}
sage: P.solve_non_stable()
\end{verbatim}
The next step is to print the output. Since it is already computed in the previous step, this is fast (although it may be lengthy, depending on the problem).

\begin{verbatim}
sage: P.print_solution_nonstable()


***************************************
SOLUTION TO GIT PROBLEM: NONSTABLE LOCI
***************************************
Group: A2
Representation  A2(3,0,0)
Set of maximal non-stable states:
(1) 1-PS = (1, 1, -2) yields a state with 7 characters
Maximal nonstable state={(1, 2, 0), (2, 1, 0), (1, 1, 1), (0, 2, 1),
(0, 3, 0), (2, 0, 1), (3, 0, 0)}
(2) 1-PS = (1, -1/2, -1/2) yields a state with 6 characters
Maximal nonstable state={(1, 2, 0), (1, 0, 2), (2, 1, 0), (1, 1, 1), (2, 0, 1),
(3, 0, 0)}
\end{verbatim}
The output above needs to be read as follows. There are 2 families (\verb|Maximal nonstable states|) that parameterise (up to $G$-conjugacy) all non-stable elements and the weights generating this family are listed. For each family there is a one-parameter subgroup $\lambda$ such that if $f\in V$ belongs to non-stable family with weights in $\Xi_{V,\lambda\geqslant 0}$, then the Hilbert-Mumford weight $\mu_(f,\lambda)\geqslant 0$ and any $\widetilde f$ with $\mu_(\widetilde f,\lambda)\geqslant 0$ must be generated by weights in $\Xi_{V,\lambda\geqslant 0}$. The output assumes that the maximal torus to which all weights belong is in diagonal form in $\mathrm{SL}_3$, so the first one-parameter subgroup must be read as $\lambda_1(t)\coloneqq\mathrm{Diag}(t^1,t^1, t^{-2})\in \mathrm{SL}_3(\mathbb C)$. Since the set of all weights corresponds to monomials of degree $3$, a non-stable $f$ in this family can be read as:
$$f=a_0x_0x_1^2+a_1x_0^2x_1+a_2x_0x_1x_2+a_3x_1^2x_2+a_4x_1^3+a_5x_0^2x_2+a_6x_0^3.$$
Sometimes the one-parameter subgroups appear as fractional vectors, such as in the second family. One can read it like an actual one-parameter subgroup by clearing denominators, i.e. $\lambda_2\coloneqq\mathrm{Diag}(t^2,t^{-1}, t^{-1})$, since positive rescaling does not affect the Hilbert-Mumford criterion.

Just as above, we can obtain the non-stable loci and the strictly $T$-polystable loci where $T$ is a maximal one-dimensional torus (recall that there is precisely one such torus up to conjugacy):
\begin{verbatim}
sage: P.solve_unstable()
sage: P.print_solution_unstable()


**************************************
SOLUTION TO GIT PROBLEM: UNSTABLE LOCI
**************************************
Group: A2 RepresentationA2(3,0,0)
Set of maximal unstable states:
(1) 1-PS = (1, 1/4, -5/4) yields a state with 5 characters
Maximal unstable state={(1, 2, 0), (2, 1, 0), (0, 3, 0), (2, 0, 1), (3, 0, 0)}

sage: P.solve_strictly_polystable()
sage: P.print_solution_strictly_polystable()


*************************************************************
SOLUTION TO GIT PROBLEM: STRICTLY POLYSTABLE LOCI
*************************************************************
Group: A2 RepresentationA2(3,0,0)
Set of strictly polystable states:
(1) A state with 1 characters
Strictly polystable state={(1, 1, 1)}
(2) A state with 3 characters
Strictly polystable state={(0, 2, 1), (2, 0, 1), (1, 1, 1)}
\end{verbatim}

\subsection{Interpretation of the outputs} 
\label{sec:comp-interpret}
The code above exemplifies how \verb|CompGIT| can provide, in terms of representation theory, a full description of the non-stable, unstable and strictly $T$-polystable families in a GIT problem. However, in applications we usually require an answer in terms that are more meaningful to the problem at hand. In the case above, for instance, we would like a classification in terms of singularities of the curves involved.

Let $p=(0:0:1)\in \mathbb P^2$. From the output in the previous section and by grouping elements, we see that nonstable elements destabilised by $\lambda_1$ can be written as
$$f=x_2\left(f_2(x_0, x_1)+f_3(x_0, x_1)\right),$$
where $f_2$ is a quadric form and $f_3$ is a cubic form linear form and any element written in that form is destabilised by $\lambda_1$. By taking partial derivatives, it follows that $f$ is of this form, if and only if its vanishing locus $V(f)$ is singular at $p$.

Similarly, nonstable elements characterised by $\lambda_2$ are precisely those that can be written as
$$f=x_0\left(f_2(x_0, x_1)+x_2l(x_0, x_1, x_2)\right),$$
where $f_2$ is a quadric form and $l$ is a linear form and any element written in that form is destabilised by $\lambda_2$. One readily checks that $f$ is of this form, if and only if its vanishing locus $V(f)$ is reducible and two of the components contain $p$.

It follows from \cite{GMGMS} that our code precisely characterises all the nonstable elements up to $G$-action. Thus we have shown that a cubic form is nonstable if and only if it is singular (as reducible curves in $\mathbb P^2$ are always singular). By taking complements, we conclude:
\begin{lemma}
    A cubic form $f$ in $3$ variables is GIT-stable if and only if $V(f)$ is  smooth.
\end{lemma}

In the case of strictly polystable loci, we have that $f$ is polystable if and only if, up to conjugation by $G$, 
$$f=x_0x_1x_2.$$
Thus:
\begin{lemma}
    A cubic form $f$ in $3$ variables is strictly GIT-polystable if and only if $V(f)$ is the union of $3$ lines which do not intersect at a point.
\end{lemma}
\begin{proof}
    Note that our output only gives us that $f$ is strictly GIT-$T$-polystable where $T$ is a maximal torus. However any strictly GIT-polystable element must be GIT-$T$-polystable and since we only have one such element and GIT-stable elements form a non-compact subset while the GIT quotient is compact, we must have at least one strictly GIT-polystable element and thus $f=x_0x_1x_2$ must be GIT-polystable.
\end{proof}

Using the programme's output, one can further characterise GIT unstable elements as those which have a singularity that is worse than an ordinary double point (i.e. an $A_2$-singularity or worse, in Arnold's language). By taking complements, one deduces that semistable elements are those which are smooth or have normal crossings. Since the only polystable element is the union of $3$ lines not intersecting at a point, any strictly semistable curve is either a nodal cubic or the union of a conic and a line intersecting normally. To fill in the details, one must make use of the theory of $A_n$ and $D_n$ singularities and their degenerations (cf. \cite[Table 3 and Fig. 2]{Gallardo-JMG-cubics}), as well as the fact that local degenerations of singularities of cubic curves are unobstructed \cite[Theorem 1]{Shustin-Tyomkin}.

\subsection{Applications} 
\label{sec:applications}
In \cite[Table 1]{GMGMS}, a number of applications to modelling moduli spaces are discussed, with references to the literature. We summarise and expand them here. We do not see this section as a description of full-worked out applications, but as a list of problems one can solve using the output of \verb|CompGIT|.
 
\subsubsection{Type $A_n$}
For $G$ of type $A_n$ (i.e. if $G=\mathrm{SL}_{n+1}$), the moduli of hypersurfaces in $\mathbb P^n$ can be modeled as $\mathbb P(V)\git \mathrm{SL}_{n+1}$ where $V=H^0(\mathbb P^n, \mathcal O_{\mathbb P^n}(d))$. The highest weight for this representation is $d\omega_1$, where $\omega_i$ denote the fundamental weights for the root system of $G$. For plane curves of degree $7$ and higher, the output has never been studied. For quintic surfaces and threefolds there are only partial results \cite{Gallardo-quintics, Gallardo-quintic-thesis, quintic-threefolds} with no results for degree $6$ or higher. Even the GIT quotient of quartic threefolds is largely unexplored and in higher dimensions we only have results for cubics.

One may consider linear systems $\mathcal L^k_d$ generated by $k$ hypersurfaces of degree $d$ in $\mathbb P^n$. Extending the action to $V$ as above , we get an action on $\Lambda_{i=1}^k V$ with highest weight $d\omega_1\wedge\cdots\wedge d\omega_1$. When $n=3, d=2, k=2$, an example is worked out in \cite{GMGMS}, cf. \cite{papazachariou-thesis, Papazachariou}.

\subsubsection{Type $B_n$}
Let $G$ be of type $B_n$. Then $G=\mathrm{SO}_{2n+1}$. Consider a smooth quadric $Q\subset \mathbb P^{2n}$. Since $Q$ is a Fano hypersurface, any automorphism of $Q$ must fix $-K_Q=\mathcal O_{\mathbb P^{2n}}(2n+1-2)|_Q\cong \mathcal O_Q(2n-1)=(2n-1)\mathcal O_Q(1)$, i.e. any automorphism of $Q$ is induced by an automorphism of $\mathbb P^{2n}$ that fixes $Q$. By diagonalising $Q$, we have that the automorphism group $\mathrm{Aut}(Q)\cong \mathrm{PO}_{2n+1}( Q)$, i.e. the projective orthogonal group fixing $Q$. But since we diagonalised $Q$,
$$\mathrm{PO}_{2n+1}(Q)=\mathrm{PO}_{2n+1}\cong \mathrm{SO}_{2n+1}=G.$$
Now let $d\geqslant 2$ consider the moduli of hypersurfaces $X\in |\mathcal O_Q(d)|=\mathbb P(V)$, where $V$ is defined by the exact sequence
$$0\longrightarrow H^0(\mathbb P^{2n}, \mathcal O_{\mathbb P^{2n}}(d-2))\longrightarrow H^0(\mathbb P^{2n}, \mathcal O_{\mathbb P^{2n}}(d)\longrightarrow V \longrightarrow 0.$$
Then the GIT quotient $\mathbb P(V)\git G$ is a model of the moduli of complete intersections of type $(2, d)$. In the case when $n=2$ and $d\geqslant 3$ (i.e. when $X\subset Q$ are surfaces of general type) the GIT stack compactifies such surfaces with at worst semi-log canonical singularities (thus providing a model for the KSBA compactification) and for $d\geqslant 4$, these surfaces are actually all in the GIT stable locus, as proven by Byun-Lee \cite{Byun-Lee-general-type}. Thus, the GIT quotient provides a model for the KSBA compactification.  Interestingly, the description of the GIT quotient in terms of the singularities of the surfaces is not provided even for $d=3$ and our program could provide the initial step to describe such quotient (i.e. not just `semi-log canonical' but the explicit singularity types that can be realised on this quotient), addressing a concern of  \cite{Byun-Lee-general-type}: ``If $S$ is a complete intersection defined by hypersurfaces with arbitrary degree in P4, then GIT stability analysis is hard to describe''. As it can be seen from Table \ref{tab:b2}, the number of families to consider is relatively small for the first few degrees, so this is a  doable problem. The notation in Table \ref{tab:b2} (and subsequent Tables \ref{tab:b3}, \ref{tab:d2}, \ref{tab:d3}) follows that in \cite{GMGMS}; namely the `Rep.' column denotes the irreducible representation $V=\Gamma_w$ determined by its highest weight $w$ and the following elements denote the size of certain sets of interest. $|\Xi_V|$ is the number of weights of the representation $V$ and serves as a measure of the complexity of the input. The numbers $|P^F_{s}|$, $|P^F_{ss}|$, $|P^F_{ps}|$ is the number of non-stable, unstable and strictly polystable families to analyse.
\small
\begin{table}[!ht]
\label{tab:b2} 
\begin{center}
\begin{tabular}{l|ll|c|rrr}

 & Type & Rep. & $|\Xi_V|$ &  $|P_s^F|$ & $|P_{ss}^F|$ & $|P_{ps}^F|$\\ \hline
$d=3$ & B2 & $\Gamma_{3\omega_1}$ & 25 &  3 &  2 &  4\\
$d=4$ & B2 & $\Gamma_{4\omega_1}$ & 41 &  4 & 3 & 5 \\
$d=5$ & B2 & $\Gamma_{5\omega_1}$ & 61 &  6 & 5 & 7 \\
$d=6$ & B2 & $\Gamma_{6\omega_1}$ & 85 &  7 & 6 & 8 \\
$d=7$ & B2 & $\Gamma_{7\omega_1}$ & 113&  10 & 9 & 11 \\
$d=8$ & B2 & $\Gamma_{8\omega_1}$ & 145 &  12 & 11 & 13 

\end{tabular}
\end{center}
\caption{Complexity of solutions for problem of type $B_2$.}
\end{table}
\normalsize
In fact, note that the GIT quotient with respect to $\mathrm{SL}_5$-action is well-understood \cite{allcock}, so comparing it to the $\mathrm{SO}_5$ quotient may prove both useful and interesting.

If we consider $n=3$ (representing fourfolds obtained as a complete intersection of a smooth quadric and a hypersurface of degree $d$ in $\mathbb P^6$), the complexity grows significantly (see Table \ref{tab:b3}) and it may not be possible to say much about it except for $d=3$. In particular, the programme is not able to obtain the strictly polystable data in a reasonable time. 
\begin{table}[!ht]
\small
\begin{center}
\begin{tabular}{l|ll|r|rrr}
 & Type & Rep. & $|\Xi_V|$ &  $|P_s^F|$ & $|P_{ss}^F|$ \\ \hline

$d=3$ & B3 & $\Gamma_{3\omega_1}$ & 14 & 8 & 7 \\
$d=4$ & B3 & $\Gamma_{4\omega_1}$ & 26 &  20 & 26 \\
$d=5$ & B3 & $\Gamma_{5\omega_1}$ & 48 &  68 & 102 \\
$d=6$ & B3 & $\Gamma_{6\omega_1}$ & 70 &  141 & 227 \\
\end{tabular}
\end{center}
\caption{Complexity of solutions for problem of type $B_3$.}
\label{tab:b3}
\end{table}
\normalsize



\subsubsection{Type $C_n$}
Consider general complex curve $C$ of genus $9$ such that $C$ has no $g^1_5$. By \cite[Theorem B]{mukai-annals}, we can embed $C$ into the symplectic Grassmannian $X=\mathrm{SpG}(3,6)\subset\mathbb P^{13}$. We have that $C$  is a transversal intersection of $X$ with a linear $P\cong \mathbb P^8\subset \mathbb P^{13}$ and any two such $C$ and $P$ are unique up to the action of $G=\mathrm{Sp}(6)$, the subgroup of $\mathrm{PGL}_6$ fixing the one-dimensional space generated by the symplectic form. Here $G$ is  a group of type $C_3$. Thus, we can see any such curve $C$ as a Zariski-dense open set of $\overline M\coloneqq\mathrm{Gr}(9,V)\git \mathrm{Sp}_6$ where $V\cong \mathbb C^{14}$ is the irreducible representation of $\mathrm{Sp}_6$ with highest weight $\omega_3$. In particular $\overline M_9$ is birationally equivalent to $\overline M$.

While the dimension of the GIT quotient is not high, the boundary is expected to be surprisingly complicated. Indeed, in \cite{GMGMS}, we showed that the unstable locus and the non-stable locus where expected to have $186$ and $142$ components. It is possible many of these are isomorphic to each other, but in any case it is a considerable enterprise to describe them explicitly. While our software is primarily created to help describe GIT quotients, this example demonstrates it is also helpful to determine a priory the viability of providing an explicit description of a moduli problem via GIT.

\subsubsection{Type $D_n$}
We recall two applications for $D_n$. For the first one, let $\mathrm{OG}(5, 10)\subset \mathbb P^{15}$ be the naturally-embedded $10$-dimensional orthogonal Grassmannian. If $C$ is a smooth genus-$7$ curve with no $g^1_4$, it has an ideal generated by $10$ quadrics in $\mathbb P^6$. Such a curve can be embedded in $\mathrm{OG}(5, 10)$ by sending each point $p\in C$ to the value of the Jacobian matrix of the quadrics, as proven by Mukai in \cite{Mukai-genus7}, where equations for the quadrics are provided, cf. \cite{Swinarski} for a more contemporaneous treatment. In fact, Mukai shows that the image of $C$ in $\mathrm{OG}(5, 10)\subset \mathbb P^{15}$  is the transversal intersection in $\mathbb P^15$ of $\mathrm{OG}(5, 10)$ with a $6$-dimensional subspace in  $\mathbb P^{15}$ and that any such transversal intersection provides a curve of genus $7$ with no $g^1_4$, with that one-to-one correspondence being unique up to action of $\mathrm{SO}_{10}$, which is a group of type $D_5$.

Since $\mathrm{Spin}_{10}$ is the universal cover of $\mathrm{SO}_{10}$, and curves with no $g^1_4$ are general in the moduli of genus $7$ curves, the GIT quotient
$$\mathrm{Gr}(7,16)\git \mathrm{Spin}_{10}\subset \mathbb P(S^+)\git \mathrm{Spin}_{10}$$
is a birational model of $\overline{M}_7$, where $S^+$ is the $16$-dimensional half-spin representation of $\mathrm{Spin}_{10}$ with highest weight $\omega_4$. As argued in \cite{GMGMS}, modern desktop computers do not seem to be able to run \verb|CompGIT| for this problem, as the output is expected to be of the order of $10^6$ families to analyse. We must emphasise that even if a supercomputer could run \verb|CompGIT| for this problem in a reasonable time (if, say, parallelisation and a smart use of memory was carried out), the output would be too complex to analyse by hand. Nonetheless, in  \cite{Mukai-genus7}, Mukai uses this construction to classify tetragonal curves of genus $7$ and in \cite{Swinarski}, uses the above analysis as a first step to further use invariant theory to establish the GIT semistability of $7$-cuspidal curves, the balanced ribbon and a family of highly reducible nodal curves.

The second application is the odd-dimensional version of what we covered in \S4.3.2. Let $n\geqslant 2$ and $G=\mathrm{SO}_{2n-1}$. Then $G$ is of type $D_n$ and, by the same reasoning, we have a natural action of $G$ on $\mathbb P(H^0(\mathbb P^{2n-1}, \mathcal{O}_{\mathbb P^{2n-1}}(d)))$, where the latter parametrises complete intersections (of dimension $2n-3$) of a smooth quadric and a hypersurface of degree $d$ in $\mathbb P^{2n-1}$ with $\mathbb P(V)\git G$ giving a model of their moduli space. In Table \ref{tab:d2} we give the complexity for curves. Since $-K_Q=\mathcal O_Q(2)$, for each $d$, this natural moduli problem parameterises the moduli of curves in $Q\cong \mathbb P^1\times\mathbb P^1$ of bidegree $(d,d)$.

\small
\begin{table}[!hb]
\small
\begin{center}
\begin{tabular}{l|ll|r|rrr}
 & Type & Rep. & $|\Xi_V|$ & $|P_s^F|$ & $|P_{ss}^F|$ & $|P_{ps}^F|$\\ \hline
$d=2$ & D2 & $\Gamma_{2\omega_1}$ & 3 & 3 & 2 & 4 \\
$d=3$ & D2 & $\Gamma_{3\omega_1}$ & 3 & 3 & 4 & 3 \\
$d=4$ & D2 & $\Gamma_{4\omega_1}$ & 5 & 5 & 4 & 6 \\
$d=5$ & D2 & $\Gamma_{5\omega_1}$ & 7 & 7 & 8 & 7 \\
$d=6$ & D2 & $\Gamma_{6\omega_1}$ & 9 & 9 & 8 & 10 \\
\end{tabular}
\end{center}
\caption{Complexity of solutions for problem of type $D_2$.}
\label{tab:d2}
\end{table}
\normalsize

The next dimension still has some interest while being accessible to classification. The GIT quotient $\mathbb P(V)\git G$ parameterises threefolds obtained as complete intersections of a smooth quadric $Q$  and a hypersurface of degree $d$ in $\mathbb P^5$.
\small
\begin{table}[!hb]
\small
\begin{center}
\begin{tabular}{l|ll|r|rrr}
 & Type & Rep. & $|\Xi_V|$ & $|P_s^F|$ & $|P_{ss}^F|$ & $|P_{ps}^F|$\\ \hline
$d=2$ & D3 & $\Gamma_{2\omega_1}$ & 7 & 4 & 2 & 6 \\
$d=3$ & D3 & $\Gamma_{3\omega_1}$ & 11 & 11 & 14 & 10 \\
$d=4$ & D3 & $\Gamma_{4\omega_1}$ & 22 & 18 & 24 & 16\\
\end{tabular}
\end{center}
\caption{Complexity of solutions for problem of type $D_3$.}
\label{tab:d3}
\end{table}
\normalsize

\subsubsection{Exceptional types}
One of the most important additions to \verb|CompGIT| since the prototype code of \cite{GMGMS} is the inclusion of exceptional Lie groups, i.e those of type $G_2$, $F_4$, $E_6$, $E_7$ and $E_8$. Explicit examples of their GIT problems are harder to find, beyond canonical examples such as their fundamental representations. Related examples come from gauge theory, in which one can take exceptional structure group $G$ and study GIT problems modulo the corresponding gauge group \cite{ramanathan_moduli_I, ramanathan_moduli_II}.  


\section{Possible future developments}
\label{sec:future-developments}
In \cite[\S 7]{GMGMS}, we detailed a number of potential improvements to the algorithms and the code introduced there, which we have completed here. It is difficult to elaborate on those potential improvements without going into the details of the implementation. Moreover, there is little we can add to what was written there and we refer the reader to \cite{GMGMS} for some ideas on generalisations to variations of GIT quotients (cf. \cite{Gallardo-JMG-seminal, GMG17} for previous work by the second author). Nonetheless, we can offer some additional changes not considered in \cite{GMGMS} that could be considered and that were not considered in \cite{GMGMS}, as well as elaborating slightly on \cite[\S 7.3 Parallel computing]{GMGMS}.

\subsection{Parallel computing}
The main computational hurdle of our algorithms is finding solutions to one-parameter subgroups. Solutions do not depend on previous computations and this is an aspect that could be easily parallelised using the \verb|sage.parallel| library. The reason we have not carried out this implementation is because for current applications, the time saved is neglectable when compared to the time to process the output in geometric terms. There is one exception where it may be worth considering parallelisation. The computation of the strictly polystable loci requires determining if a given point is in the interior or the boundary of a convex hull of points. The current implementation uses the class \verb|Polyhedron| --- \verb|sage.geometry.polyhedron.constructor| --- as a black box to obtain this information. When the number of points to consider is large, the current implementation takes too long, even for a relatively small number of families (e.g. for only $8$ families in the case of $B_4$ acting on $\Gamma_{\omega_3}$, the computation was stopped after 48 hours on a desktop computer, see \cite[Table 1]{GMGMS}). 

\subsection{Generalisation to semisimple groups and other groups}
Our implementation works for irreducible $G$-representations where $G$ is a simple group, while our algorithms (once some optimisations detailed in \cite{GMGMS} are removed) works for any reductive algebraic group. The reason for this is that the implementation relies on two pre-existing classes in \verb|Sagemath|, namely \verb|WeylCharacterRing| and \verb|WeylGroup|.

\verb|WeylCharacterRing| should work for general semisimple groups without changing the code. This class allows us to take representations as a black box, but the only bit of this class that we use is the method \verb|weight_multiplicities| in \verb|Git.py| within the constructor of the class \verb|GITProblem| to obtain a list of non-zero weights of the representation. Our implementation could easily be modified to request from the user a list of the weights of a representation, thus allowing us to work with reducible representations and more general reductive groups.

\verb|WeylGroup| is used within \verb|SimpleGroup.py| to obtain a list of all the elements of the Weyl group of the simple group $G$ to potentially simplify the output and reduce the number of computations (this is what we call \emph{Weyl optimisation} in \cite{GMGMS}). To generalise to more general reductive groups, some additional work must be carried out here. There are two options. The first option is to describe the Weyl group for more general $G$. We do not recommend this, as it will be highly dependent on the choice of group, making it difficult to provide a uniform approach. A simpler option is to stop carrying out any Weyl optimisation at all. The solution methods (\verb|solve_non_stable|, \verb|solve_unstable|) already include a parameter (\verb|Weyl_optimisation|), set by default to \verb|False|, not to carry out this optimisation. However, the constructors assume that the group $G$ is simple and the Weyl group is computed nonetheless.

It is worth noting that it is actually believed that the output is not simplified at all by the use of the Weyl group (see \cite[Conjecture 7.4]{GMGMS}), although the programme may be faster when Weyl optimisation is used.

\subsection{Non-reductive groups}
As exemplified by Nagata \cite{counter-example}, the GIT of Mumford (and more generally invariant theory) does not work well for non-reductive groups. The more recent theory of \emph{non-reductive GIT} due to Kirwan and her collaborators \cite{non-reductive-GIT} does allow for GIT quotients by non-reductive groups under an additional choice of \emph{graded linearisation}. To the best of our knowledge no systematic algorithmical study has been carried out, although some seminal work for hypersurfaces has been considered \cite{Dominic}. A non-reductive version of our setting (a non-reductive group $G$ acting on a vector space $V$) is a natural starting point.

\subsection{Generalisations to other schemes}
As discussed in \S \ref{sec:GIT quotients} a $G$-equivariant embedding of a $G$-scheme $X$ into $\mathbb P^N$ is necessary to apply the outputs of our code. Implementing solutions to this embedding problem into {\verb|CompGIT|} appears to be out of reach, as for general $X$, one does not expect a uniform computational approach to finding such an embedding. Nonetheless, it may be possible to give an ad-hoc treatment for some schemes of geometric interest. For instance, one can consider Grassmannians, or hypersurfaces, or products and complete intersections of either of the former (note that $\mathbb P^N$ is a special case of Grassmannian). Indeed, some seminal work for $G=\mathrm{SL}_{n+1}$ acting on $V=H^0(\mathbb P^n, \mathcal O(d))\times\mathbb P^n$ naturally was carried out in \cite{Gallardo-JMG-seminal, Karagiorgis-Ortscheidt-Papazachariou} (including code \cite{GMG17}) and quotients for $\mathrm{SL}_n$ actions on a Grassmannian and on a product of a Grassmannian and projective space were considered in \cite{papazachariou-thesis, Papazachariou}.

\bibliographystyle{alpha} 
\bibliography{main} 


\end{document}